\documentclass[11pt,a4paper,usenames,dvipsnames,reqno]{amsart}

\usepackage{url,amsmath,amssymb,latexsym,mathrsfs,comment,amsthm,enumerate,geometry,calc,ifthen,mathdots,textcomp}

\usepackage{cases}

\usepackage[noadjust]{cite}

\usepackage[colorlinks]{hyperref}

\usepackage[margin=10pt,font=small,labelfont=bf, labelsep=period]{caption}

\usepackage{makecell}

\usepackage[x11names, svgnames, rgb,table]{xcolor}
\usepackage{hhline}
\usepackage{multirow}

\usepackage{enumitem}

\usepackage[utf8]{inputenc}
\usepackage[T1]{fontenc}

\usepackage{tikz}
\usetikzlibrary{decorations.markings}
\usetikzlibrary{matrix}
\usepgflibrary{snakes,arrows,shapes}
\pgfdeclarelayer{background layer}
\pgfsetlayers{background layer,main}

\newcommand\nc\newcommand
\nc\rnc\renewcommand

\nc\Sub{\operatorname{\sf Sub}}
\nc\NSub{\operatorname{\sf NSub}}
\nc\Ht{\operatorname{\sf Ht}}
\nc\bH{{\bf H}}
\nc\es\varnothing
\nc\la\langle
\nc\ra\rangle
\nc\G{\mathcal G}
\nc\rG{\mathrel{\G}}
\nc\leqG{\operatorname{\leq_\G}}
\nc{\Mod}[1]{\ (\mathrm{mod}\ #1)}

\DeclareMathOperator{\kerr}{\overline{\text{ker}}}

\nc\trans[1]{\left(\begin{smallmatrix}#1\end{smallmatrix}\right)}

\nc\ol\overline

\newcommand{\Cong}{\operatorname{\sf Cong}}

\nc\CongS{\Cong^S\hspace{-0.6truemm}}

\nc\al\alpha
\nc\be\beta
\nc\ga\gamma
\nc\de\delta
\rnc\th\theta
\nc\lam\lambda
\nc\om\omega
\nc\ze\zeta
\nc\si\sigma
\nc\Si\Sigma
\nc\Ga\Gamma
\nc\De\Delta
\nc\Th\Theta
\nc\Om\Omega
\nc\nab\nabla

\nc\AND{\qquad\text{and}\qquad}
\nc\ANDSIM{\qquad\text{and similarly}\qquad}
\nc\ANd{\quad\text{and}\quad}
\nc\COMMA{,\qquad}
\nc\COMMa{,\quad}
\nc\WHERE{\qquad\text{where}\qquad}

\rnc\iff{\ \Leftrightarrow\ }
\nc\IFf{\quad \Leftrightarrow\quad }
\nc\Iff{\ \ \Leftrightarrow\ \ }
\nc\IFF{\qquad \Leftrightarrow\qquad }
\rnc\implies{\ \Rightarrow\ }
\nc\IMPLIES{\qquad \Rightarrow\qquad }

\nc\set[2]{\{#1:#2\}}
\nc\bigset[2]{\big\{#1:#2\big\}}
\nc\pres[2]{\la#1:#2\ra}

\nc\bit{\begin{itemize}[label=\textbullet, leftmargin=5mm]}
\nc\eit{\end{itemize}}
\nc\ben{\begin{thmenumerate}}
\nc\bena{\begin{enumerate}[label=\textup{(\alph*)},leftmargin=10mm]}
\nc\een{\end{thmenumerate}}
\nc\eena{\end{enumerate}}

\nc\pf{\begin{proof}}
\nc\epf{\end{proof}}
\nc\pfclaim{\begin{quote}\begin{proof}}
\nc\epfclaim{\end{proof}\end{quote}}
\nc\epfres{\hfill\qed}
\nc\epfreseq{\tag*{\qed}}
\let\oldproofname=\proofname
\renewcommand{\proofname}{\rm\bf{\oldproofname}}
\nc{\pfitem}[1]{\medskip \noindent #1.}
\nc{\firstpfitem}[1]{#1.}
\nc{\pfcase}[1]{\medskip\noindent {\bf Case #1.}}

\renewcommand{\H}{\mathcal H}
\renewcommand{\L}{\mathcal L}
\newcommand{\R}{\mathcal R}
\newcommand{\D}{\mathcal D}
\newcommand{\J}{\mathcal J}
\newcommand{\K}{\mathcal K}
\nc\rH{\mathrel{\H}}
\nc\rL{\mathrel{\L}}
\nc\rR{\mathrel{\R}}
\nc\rD{\mathrel{\D}}
\nc\rJ{\mathrel{\J}}
\nc\rK{\mathrel{\K}}
\nc\rsi{\mathrel{\si}}

\nc\leqL{\leq_\L}
\nc\leqR{\leq_\R}
\nc\leqJ{\leq_\J}
\nc\leqH{\leq_\H}

\renewcommand{\P}{\mathcal P} 

\rnc\S{\mathcal S}
\nc\A{\mathcal A}
\nc\B{\mathcal B}
\nc\M{\mathcal M}
\nc\TL{\mathcal T\!\mathcal L}
\nc\F{\mathcal F}
\nc\I{\mathcal I}
\nc\C{\mathcal C}
\nc\POI{\mathcal{POI}}
\nc\PO{\mathcal{PO}}
\rnc\O{\mathcal O}
\nc\LL{\mathcal L}
\newcommand{\T}{\mathcal T}
\newcommand{\PT}{\mathcal P\mathcal T} 
 
\renewcommand{\I}{\mathcal I}  

\nc\bs[1]{\boldsymbol{#1}}

\newcommand{\multwk}[1]{\cdot}

\nc\permdiff\partial
\nc\pd\permdiff
\nc\Proj[1]{\overline{#1}}
\nc\norm[1]{\langle\!\langle#1\rangle\!\rangle}
\nc\fd\De

\newcommand{\N}{\mathbb{N}}
\nc\bn{{\bf n}}
\nc\bm{{\bf m}}
\nc\bnz{{\bf n}_0}
\nc\bmz{{\bf m}_0}
\nc\bdz{{\bf d}_0}
\nc\bq{{\bf q}}
\nc\zero{{\bf 0}}

\newcommand{\coker}{\operatorname{coker}}
\newcommand{\dom}{\operatorname{dom}} 
\newcommand{\codom}{\operatorname{codom}}

\newcommand{\im}{\operatorname{im}}
\let\ker\undefined
\newcommand{\ker}{\operatorname{ker}}

\nc\thx{\theta^{\textsf{x}}}

\nc\Pair\Pi

\makeatletter
\DeclareRobustCommand
  \vvdots{\vbox{\baselineskip4\p@ \lineskiplimit\z@\kern4\p@
    \hbox{.}\hbox{.}\hbox{.}}}
\makeatother

\rnc\emptyset{\varnothing}
\nc\sub\subseteq
\nc\mt\mapsto
\nc\wh\widehat
\nc\normal\unlhd
\nc\sm\setminus
\nc\mr\mathrel
\nc\pc[2]{(#1,#2)^\sharp}

\nc{\uv}[1]{\fill (#1,2)circle(.17);}
\nc{\lv}[1]{\fill (#1,0)circle(.17);}
\nc{\uvs}[1]{{\foreach \x in {#1} { \uv{\x}}}}
\nc{\lvs}[1]{{\foreach \x in {#1} { \lv{\x}}}}
\nc{\darcx}[3]{\draw(#1,0)arc(180:90:#3) (#1+#3,#3)--(#2-#3,#3) (#2-#3,#3) arc(90:0:#3);}
\nc{\darc}[2]{\darcx{#1}{#2}{.4}}
\nc{\uarcx}[3]{\draw(#1,2)arc(180:270:#3) (#1+#3,2-#3)--(#2-#3,2-#3) (#2-#3,2-#3) arc(270:360:#3);}
\nc{\uarc}[2]{\uarcx{#1}{#2}{.4}}
\nc{\stline}[2]{\draw(#1,2)--(#2,0);}

\newtheorem{thm}{Theorem}[section]
\newtheorem*{thm*}{Theorem}
\newtheorem*{theorem*}{Main theorem}

\newtheorem{lemma}[thm]{Lemma}
\newtheorem{cor}[thm]{Corollary}
\newtheorem{prop}[thm]{Proposition}

\theoremstyle{definition}

\newtheorem{rem}[thm]{Remark}

\newenvironment{thmenumerate}{
   \begin{enumerate}[label=\textup{(\roman*)}, widest=(5), leftmargin=10mm]}{
    \end{enumerate}}

\newcounter{caseco}

\newcounter{subcaseco}

\newcounter{stepco}

\newcounter{stageco}

\newcounter{RT}\renewcommand{\theRT}{RT\arabic{RT}}

\newcounter{fR}\renewcommand{\thefR}{fR\arabic{fR}}

\definecolor{delcol}{HTML}{F4F3F4}
\definecolor{excepcol}{HTML}{D50B53}
\definecolor{mucol}{HTML}{B9C406}
\definecolor{Ncol}{HTML}{FAA565}
\definecolor{Rcol}{HTML}{A882C1}

\usepackage{comment}

\newcommand{\NZ}{\mathbb{N}_0}

\newcommand{\Z}{\mathbb{Z}}

\newcommand{\Lg}{\mathcal{L}}

\newenvironment{nitemize}{\begin{itemize}[label=\textbullet, leftmargin=5mm]}{\end{itemize}}

\begin{document}

\title[Coherency for monoids]{Coherency properties for monoids of transformations and partitions}
\date{\today}
\author{Matthew Brookes}
\address{Mathematical Institute, School of Mathematics and Statistics, University of St Andrews, St Andrews, Fife KY16 9SS, UK.}
\email{mdgkb1@st-andrews.ac.uk}

\author{Victoria Gould}
\address{Department of Mathematics, University of York, York YO10 5DD, UK.}
\email{victoria.gould@york.ac.uk}

\author{Nik Ru\v{s}kuc}
\address{Mathematical Institute, School of Mathematics and Statistics, University of St Andrews, St Andrews, Fife KY16 9SS, UK.}
\email{nik.ruskuc@st-andrews.ac.uk}

\thanks{Supported by the Engineering and Physical Sciences Research Council [EP/V002953/1, EP/V003224/1]}

\keywords{Monoid, Coherency, Right congruence, Transformation monoid, Partition monoid}

\subjclass[2020]{20M30, 20M05, 20M18, 20M20}

\maketitle


\begin{abstract}
\noindent
A monoid $S$ is {\em right coherent} if every finitely generated
subact of every finitely presented right $S$-act itself has a finite presentation; it is \emph{weakly right coherent} if every 
finitely generated right ideal of $S$ has a finite presentation.
We show that full and partial transformation monoids, symmetric inverse monoids and partition monoids over an infinite set are all weakly right coherent,
but that none of them is right coherent.
Left coherency and weak left coherency are defined dually, and the corresponding results hold  for these properties.
In order to prove the non-coherency results, we give a presentation of an inverse semigroup which does not embed into any left or right coherent monoid.
\end{abstract}


\section{Introduction}
\label{sec:intro}

A monoid $S$ is {\em right coherent} if every finitely generated
subact of every finitely presented right $S$-act itself has a finite presentation. The concept has several natural sources, including the corresponding notion for rings \cite{chase:1960},
the model theoretic notion of model companion \cite{wheeler:1976} and questions of axiomatisability of classes of existentially closed or algebraically closed $S$-acts \cite{gould:1986}. It is also intimately connected with notions of injectivity and purity for $S$-acts \cite{gould:2023}.

A monoid $S$ is {\em weakly right coherent} if every finitely generated
right ideal of $S$ has a finite presentation. For rings the corresponding notions of right coherency and weak right coherency coincide, essentially because every cyclic right module is determined by a right ideal. This is not the situation for monoids. Not every cyclic right $S$-act is determined by a right ideal: one must consider right congruences per se. 

We remark that other notions of coherency exist in the literature. For example, a group is coherent if every finitely generated subgroup has a finite presentation \cite{wise:2019}.  More broadly, one can say that a (universal) algebra is coherent if every finitely generated subalgebra has a finite presentation. Thus, a monoid $S$ is right coherent in our sense if every finitely presented right $S$-act is coherent in this latter sense.  

Left coherency and weak left coherency are defined dually to the right-handed concepts.
Left and right coherency, and hence left and right weak  coherency, are finitary conditions for monoids in the sense that any finite monoid satisfies them. These notions were introduced in 
\cite{gould:1986} and developed in a number of articles.   Groups, right noetherian monoids, free monoids and regular monoids in which all right ideals are finitely generated are right  coherent \cite{gould:1992,gould:2017,dandan:2020}. On the other hand, free inverse monoids  of rank $\geq 2$ are weakly left and right coherent but neither left nor right coherent \cite{gould:2017a}.  The pattern of these results does not guide us in drawing immediate conjectures with respect to some very natural infinite monoids of transformations and partitions.  
The aim of this paper is to remedy this situation. In particular, we resolve the question of (weak) left and right coherency for a number of monoids of major interest. The following is an aggregate of Theorems~\ref{thm:Ix Tx PTx WRC}, \ref{thm:Ix Tx PTx WLC} and \ref{thm: Px WRC} (for weak coherency), and Theorem~\ref{thm:coh} (for coherency).

\begin{theorem*}  Let $X$ be an infinite set. The  full transformation monoid
$\mathcal{T}_X$,  the partial transformation monoid $\mathcal{PT}_X$, the symmetric inverse monoid $\mathcal{I}_X$ and the partition monoid $\mathcal{P}_X$ on  $X$, are all weakly left and right coherent but neither left nor right coherent.
\end{theorem*} 

In order to prove the non-coherency results, we develop a new piece of methodology. Namely, we give  a presentation of a monoid $P(g,h,e)$ 
such that if $P(g,h,e)$ occurs as a  subsemigroup of any monoid $M$, it prevents $M$ from being right or left coherent.

We summarise the concepts required for this paper in Section~\ref{sec:prelim}. In 
Section \ref{sec:weak for partition} we prove the positive results for weak left and right coherency. In 
Section \ref{sec:annihilators} we introduce the monoid $P(g,h,e)$ and show that it cannot be a subsemigroup of any left or right coherent monoid. 
We observe that $P(g,h,e)$ is isomorphic to a submonoid of $\I_\Z$. Having done the brunt of the work, in Section \ref{sec:examples}
we are rapidly able to prove our non-coherency results.

\section{Preliminaries}
\label{sec:prelim}

\subsection{Monoids}
Throughout this section, unless otherwise stated, $S$ will stand for a monoid, and its identity will be denoted by $1_S$.
We use $E(S)$ to denote the set of idempotents of $S$. A subset $T\subseteq S$ is a \emph{subsemigroup} of $S$ if it is closed under the multiplication, and a \emph{submonoid} if additionally $1_S\in T$.
A \emph{right ideal} of $S$ is a subset $I\subseteq S$ such that $IS:=\{xs\mid x\in I,s\in S\} \subseteq I$.

\subsection{Right congruences}
For a relation $\rho\subseteq X\times X$ on a set $X$, we will use both $(x,y)\in\rho$ and $x\,\rho\, y$ to mean that $x$ and $y$ are $\rho$-related. We also write $\rho^{-1}$ for the relation $\{ (x,y)\mid (y,x)\in\rho\}$.

A \emph{right congruence} on a monoid $S$ is an equivalence relation $\rho\subseteq S\times S$ which is \emph{right compatible}, i.e. satisfies $a\,\rho\, b \Rightarrow ac\,\rho\, bc$ for all $a,b,c\in S$. We denote by $\Delta$ the equality relation on $S$, which is the smallest right congruence on $S$.
For a set $Y\subseteq S\times S$ we denote by $\langle Y\rangle$ the right congruence on $S$ \emph{generated} by $Y$,
which is the smallest right congruence that contains $Y$.
The following criterion for a pair to belong to $\langle Y\rangle$ will be used throughout the paper: for $a,b\in S$ we have
$(a,b)\in\langle Y\rangle$ if and only if there exists a $Y$-sequence from $a$ to $b$, i.e. a sequence of the form
\[
a=c_1t_1,\, d_1t_1=c_2t_2,\cdots, d_mt_m=b,
\]
where $m\geq 0$, $(c_i,d_i)\in Y\cup Y^{-1}$, $t_i\in S$, $i=1,\dots,m$.
The number $m$ is called the \emph{length} of the sequence. The case $m=0$ is interpreted as $a=b$.
For details see \cite[Lemma 1.4.37]{kilp00}. 
A right congruence $\rho$ is {\em finitely generated} if there is a finite set $Y$ such that $\rho=\langle Y\rangle$. Equivalently, $\rho$ is  finitely generated precisely when there is no infinite strictly ascending chain of right congruences $\rho_1\subset \rho_2\subset\, \dots$ such that $\rho=\bigcup_{n\in \N} \rho_n$.

\subsection{Right $S$-acts}
A {\em right $S$-act} is a set $A$ together with a function $A\times S\rightarrow S$, $(a,s)\mapsto as$, such that $(as)t=a(st)$ and $a1_S=a$ for all $a\in A$ and all $s,t\in S$. 
For a set $X\subseteq A$ we write $XS:=\{xs\mid x\in X,s\in S\}$.
A subset $B\subseteq A$ is an \emph{subact} of $A$ if $BS\subseteq B$.
The $S$-act $A$ is said to be \emph{generated} by a set $X\subseteq A$ if $XS=A$,
{\em finitely generated} if it is generated by a finite set, and  {\em monogenic} if it is generated by a single element.
For a more detailed introduction to actions see \cite[Section 1.4]{kilp00}.

The right ideals of $S$ are  right $S$-acts via right multiplication by the elements from $S$.
They are precisely the subacts of the (monogenic) right $S$-act $S$.
If $\rho$ is a right congruence on $S$ then the \emph{quotient act} $S/\rho$ is the set of equivalence classes together with the action $(a\rho)s=(as)\rho$. It is well known that there is a natural bijective identification between monogenic $S$-acts and right congruences; see \cite[Theorem 11.6]{clifford67}.

\subsection{Right coherence}
We saw in the Introduction how the notions of right coherence and weak right coherence both naturally arise from the notion of right coherence for rings. In addition to the definitions given in the Introduction, in terms of finite generation and presentability of $S$-acts, both these notions have a number of additional equivalent formulations. Here we present those that are most convenient for our purposes.
For this we need the notion of the {\em right annihilator} congruence of an element $a\in S$ with respect to a right congruence $\rho$:
    $$r(a\rho):=\{(u,v)\in S\times S \mid au\ \rho\ av\}.$$
It is straightforward to check that $r(a\rho)$ is indeed a right congruence.

\begin{thm*}[{\cite[Corollaries 3.3, 3.4]{gould:1992}}] Let $S$ be a monoid.
\begin{enumerate}[label=\textup{(\arabic{section}.\arabic{subsection}.\arabic*)},leftmargin=13mm]
\item
\label{it:cc1}
$S$ is right coherent if and only if  for all $a,b\in S$ and for all finitely generated right congruences $\rho$ the subact $(a\rho)S\cap (b\rho)S$ of the right $S$-act $S/\rho$ is finitely generated and the right annihilator $r(a\rho)$ is finitely generated. 
\item
\label{it:cc2}
$S$ is weakly right coherent if and only if for all $a,b\in S$ the right ideal $aS\cap bS$ of $S$ is finitely generated (as a subact) and the right annihilator $r(a\Delta)$ is finitely generated.
\end{enumerate}
\end{thm*}

\subsection{Green's relations} For a monoid $S$, 
 \emph{Green's preorder} $\leq_\R$  is defined by  $a\leq_\R b$ if there is an $s\in S$ such that $bs=a.$
{\em Green's equivalence} $\R$ is the equivalence relation  induced by the preorder $\leq_\R$,
i.e.  $a\ \R\ b$ if there are $s,t\in S$ such that $as=b$ and $bt=a$.
 For details see \cite[Section 2.1]{howie95}. 
 
 \subsection{Left-right duality}
 All the notions that have been prefixed by the word `right' have obvious left-right duals.
 Thus we have left congruences, left $S$-acts, weak left coherence, left coherence, left annihilators, Green's preorder $\leq_\Lg$ and Green's equivalence~$\Lg$.

\subsection{Idempotents} \label{ss:ids}
Green's preorder $\leq_\R$ is related to the natural partial order on the set of idempotents $E(S)$, which is defined by $e\leq f$ if and only if $fe=ef=e$. 
Clearly $e\leq f$ implies $e\leq_\R f$. Since $\leq$ is defined purely in terms of the two idempotents involved, for a subsemigroup $T\leq S$ we have $e\leq f$ in $T$ if and only if $e\leq f$ in $S$.  
The analogue is not true in general for the $\leq_\R$ preorder.

\subsection{Regular and inverse monoids} \label{ss:reg}
An element $a\in S$ is said to be {\em regular} if there is $y\in S$ such that $aya=a$ and $S$ is called {\em regular} if every $a\in S$ is regular. 
In this case $ay$ and $ya$ are idempotents and we have
$ay\,\R\, a\, \Lg\, ya$.
In fact, $S$ is regular precisely when every $\R$-class contains an idempotent.
Furthermore, a regular $a\in S$ has a \emph{generalised inverse} $x$, which satisfies $axa=a$ and $xax=x$.
If every element has a \emph{unique} generalised inverse, $S$ is said to be an \emph{inverse monoid}; in this case we denote the unique generalised inverse of $a$ by $a^{-1}$.
Equivalently, inverse monoids are precisely the regular monoids in which the idempotents form a commutative subsemigroup. For details see \cite[Sections 2.3, 2.4, 5.1]{howie95}.
The link between the partial order $\leq$ on idempotents, and the Green's preorder $\leq_\R$ becomes tight for inverse semigroups, in the sense that $a\leq_\R b\Leftrightarrow aa^{-1}\leq bb^{-1}$; see \cite[Proposition 2.4.1]{howie95} and its proof.
If $T$ is a regular subsemigroup of a monoid $S$ then the Green's preorders $\leq_\R$ and $\leq_\Lg$ on $T$ are
equal to the restrictions to $T$ of the preorders $\leq_\R$ and $\leq_\Lg$ on $S$; see \cite[Proposition 2.4.2]{howie95}.

\subsection{Transformation monoids}\label{ss:tran}
Let $X$ be a set. 
The {\em partial transformation monoid} $\PT_X$ is the set of partial mappings from $X$ to itself under composition of
 mappings.
For $\al\in\PT_X$, we define its \emph{domain}, \emph{image} and \emph{kernels} as follows:
\begin{align*}
\dom \al & := \{ x\in X\mid x\al \text{ is defined}\}, &
\im \al & :=\{ x\al\mid x\in X\},\\
\ker \al & :=\{ (x,y)\in \dom \al\times\dom \al\mid x\al=y\al\}, & \kerr \alpha & := \ker\alpha \cup \{ (x,y)\::\: x,y\in X\setminus \dom\alpha\}.
\end{align*}

The \emph{full transformation monoid} $\T_X$ and the \emph{symmetric inverse monoid} $\I_X$ on $X$ consist respectively of full and injective mappings from $\PT_X$:
     \[
     \T_X=\{\al\in\PT_X\mid \dom \al=X\},\quad
        \I_X=\{\al\in \PT_X\mid  \ker \alpha=\Delta_{\dom \al}\}.
     \]
It is elementary that each of $\T_X,\PT_X$ and $\I_X$ is regular, and indeed $\I_X$ is inverse. 
Therefore to describe Green's relations on all three it suffices to state them for $\PT_X$ and then appeal to \ref{ss:reg}.
Here is such a description:
    \[
        \al\leq_\Lg \be  \iff  \im \al \subseteq \im \be,\quad
        \al\leq_\R \be  \iff  \dom \al \subseteq \dom \be \text{ and } \kerr \be\subseteq \kerr \al.
   \]
For more details on transformation semigroups we refer the reader to \cite{gan09}, though do note that while stated for finite $X$ this assumption is not used. Also, the reader should note that in \cite{gan09} mappings are written to the left of arguments, and the relation $\kerr\alpha$ is denoted by $\pi_\alpha$.

\subsection{Partition monoids}\label{ss:part}
Let $X$ be a set, and let $X^\prime=\{x'\mid x\in X\}$ be a disjoint copy of $X$. The {\em partition monoid} on $X$, denoted $\P_X$,  consists of all set partitions of $X\cup X^\prime$ under the composition defined as follows. Each partition in $\P_X$ is identified with any graph on $X\cup X^\prime$ whose connected components are the blocks of the partition. These graphs are drawn with the vertices in two rows, the vertices from $X$ in the top row and those from $X^\prime$ in the bottom row.  Given $\al,\be\in \P_X$ let $\al_\downarrow$  and $\be^\uparrow$ be the graph obtained by relabelling every lower vertex $x'$ in $\al$ and every upper vertex $x$ in $\be$ by $x^{\prime\prime}$. Then let $\Pi(\al,\be)$ be the graph on $\{x,x',x''\mid x\in X\}$ with edge set the union of the edge sets of $\al_\downarrow$ and $\be^\uparrow$. The product $\al\be\in \P_X$ is then the partition such that the blocks are the subsets of $X\cup X^\prime$ which are in the same connected component of $\Pi(\al,\be).$  The mapping $\alpha\mapsto\alpha^\ast$, where $\alpha^\ast$ is obtained by swapping the top and bottom rows of $\alpha$, is an anti-isomorphism. Furthermore, $\alpha^\ast$ is a generalised inverse of $\alpha$, and hence $\P_X$ is regular.
The product in $\P_X$ and the anti-isomorphism $\alpha\mapsto\alpha^\ast$ may be nicely visualised in terms of graphs $\Pi(\alpha,\beta)$, which may help a reader with subsequent arguments, for example see \cite{east11}.

Let $\al\in \P_X$. A block of $\al$ is said to be \emph{upper} if it is contained in $X$; \emph{lower} if it is contained in $X'$; or \emph{transversal} if it contains elements from both $X$ and $X'$.
Green's preorders on $\P_X$ can be described in terms of the following parameters:
\begin{nitemize}
\item
$\dom \al$ (resp. $\codom \al$) consisting of all $x\in X$ (resp. $x'\in X')$ that belong to a transversal of $\al$;
\item
$\ker \al$ (resp. $\coker \al$), the partition on $X$ (resp. $X'$) induced by $\al$;
\item
$N_U(\al)$ (resp. $N_L(\al)$) the set of all upper (resp. lower) blocks of $\al$.
\end{nitemize}
With this notation, for $\al,\be\in\P_X$ we have (\cite[Theorem 4.4]{dolinka21}):
\begin{align*}
&\al\leq_\R \be \Leftrightarrow \ker \be\subseteq \ker \al \text{ and } N_U(\be)\subseteq N_U(\al),\\
&\al\leq_\Lg \be \Leftrightarrow \coker \be\subseteq \coker \al \text{ and } N_L(\be)\subseteq N_L(\al).
\end{align*}


\section{Weak coherency for transformation and partition monoids}
\label{sec:weak for partition}
 
In this section we prove that each of the monoids $\I_X$, $\T_X$, $\PT_X$ and $\P_X$ is weakly right and left coherent.
We begin with some general observations, which reduce the work needed to check the conditions from
\ref{it:cc2}.

\begin{lemma}\label{lem:r cong gen by (1,g)}
    Let  $S$ be  a monoid and let $s\in S$. The right congruence $\langle (1,s)\rangle$ is given by
    $$\kappa=\{(u,v)\in S\times S\mid s^mu=s^nv \text{ for some } m,n\in \NZ\}. $$ 
\end{lemma}

\begin{proof}
    It is easy to see that $\kappa$ is reflexive, symmetric and right compatible. Further, if $s^nu=s^mv$ and $s^kv=s^lw$ where $m,n,k,l\in \NZ$ then $s^{n+k}u=s^{m+l}w$ so $\kappa$ is a right congruence, which evidently contains $(1,s)$. Conversely, note that $(1,s^k)\in\langle (1,s)\rangle$ for any $k\in \NZ$. It then easily follows that if $s^nu=s^mv$ as above, then $(u,v)\in \langle(1,s)\rangle$.
\end{proof}

Note that in Lemma~\ref{lem:r cong gen by (1,g)} if $s$ is idempotent then $\kappa$ is equal to $r(e\Delta)$.

\begin{cor}\label{cor:idem}  Let  $S$ be  a monoid and let $e\in E(S)$. Then  $r(e\Delta)=\langle (1,e)\rangle$.
\end{cor}

\begin{cor}
 \label{pro:regwc}
   A regular monoid $S$ is weakly right (resp. left) coherent if and only if $aS\cap bS$ (resp. $Sa\cap Sb$) is finitely generated for all $a,b\in S$.
\end{cor}

\begin{proof}
We prove the first assertion; the second is dual.
By \ref{it:cc2}, it is sufficient to show that, when $S$ is regular, each right annihilator $r(a\Delta)$ is finitely generated.
Let $b\in S$ be such that $aba=a$, and let $e=ba\in E(S)$. 
Note that for any $x,y\in S$ we have
\[
ax=ay\Rightarrow bax=bay \Rightarrow abax=abay\Leftrightarrow ax=ay,
\]
from which $r(a\Delta)=r(e\Delta)$.  The result now follows from Corollary~\ref{cor:idem}.
\end{proof}

The condition that $aS\cap bS$  is finitely generated for all $a,b\in S$ is by virtue of distributivity equivalent to  the condition that the intersection of finitely generated right ideals is finitely generated. Monoids satifying this condition are the focus of \cite{carson:2021} where they are called  {\em right ideal Howson}; under the name of {\em finitely (right) aligned} they play a role in the theory of $C^*$-al\-geb\-ras~\cite{spielberg:2014}.
 
\begin{thm}\label{thm:Ix Tx PTx WRC}
    The monoids $\I_X$, $\T_X$ and $\PT_X$ are all weakly right coherent.
\end{thm}

\begin{proof}
    As each of $\I_X$, $\T_X$ or $\PT_X$ is regular, by Corollary \ref{pro:regwc} it suffices to check that the intersection, $\al S\cap \be S$ for $S=\I_X,\T_X$ or $\PT_X$, is finitely generated.  We give a detailed proof for $S=\PT_X$, and then outline how to adapt it for $S=\T_X$ and $S=\I_X$.
    
Recall from \ref{ss:reg} and \ref{ss:tran} that for $\mu\in S$ we have
\begin{equation}
\label{eq:alS}
 \mu S = \{\nu\in S\mid \dom \nu\subseteq \dom \mu,\ \kerr \mu \subseteq \kerr \nu\}.
\end{equation}
Also, for any $\mu,\nu \in S$ with $\dom\nu\subseteq \dom\mu$ we have
\begin{equation}
\label{eq:munu}
\kerr\mu\subseteq \kerr\nu\ \Leftrightarrow \text{ every } \ker\nu \text{-class is a union of } \ker\mu \text{-classes}.
\end{equation}

Now consider arbitrary $\al,\be\in S$.  
   Let $Y$ be the union of the equivalence classes of $\ker \al\vee \ker \be$ contained in $\dom \al\cap \dom \be$. Note that $Y$ may be empty.
    Define $\epsilon\in S$ to be any element with $\dom \epsilon=Y$ and $\ker \epsilon= (\ker \al\vee \ker \be)|_Y$.
   
 We claim that
$\alpha S\cap \beta S=\epsilon S$.
The reverse inclusion ($\supseteq$) follows by observing that $\dom\epsilon=Y\subseteq\dom\alpha\cap\dom\beta$, and that
$\ker\epsilon$ classes are unions of both $\ker\alpha$-classes and $\ker\beta$-classes, and then appealing
to \eqref{eq:munu} and then \eqref{eq:alS} to deduce $\epsilon\in \alpha S\cap \beta S$.
For the direct inclusion ($\subseteq$), suppose $\delta\in\alpha S\cap \beta S$.
By \eqref{eq:alS} we have
\[
\dom\delta\subseteq\dom\alpha,\ \dom\delta\subseteq \dom\beta,\ 
\kerr\alpha\subseteq\kerr\delta,\ \kerr\beta\subseteq\kerr\beta.
\]
From this it immediately follows that $\dom\delta\subseteq \dom\alpha\cap\dom\beta$, and that every $\ker\delta$-class is a union of $(\ker\alpha\vee\ker\beta)$-classes.
Since $Y$ is by definition the largest subset of $\dom\alpha\cap\dom\beta$ that is a union of $(\ker\alpha\vee\ker\beta)$-classes, we deduce $\dom\delta\subseteq Y=\dom\epsilon$.
Now, using \eqref{eq:munu} we have that $\kerr\delta\subseteq\kerr\epsilon$, and therefore $\delta\in \epsilon S$ by \eqref{eq:alS}, proving that $\alpha S\cap \beta S=\epsilon S$ as claimed, and hence the theorem for $S=\PT_X$.

    For $S=\T_X$ and $S=\I_X$ the same proof works, but with some conditions concerning domains and kernels omitted. Specifically, as $\dom \al=X$ for every $\al\in \T_X$, all requirements concerning domains become vacuous in this case, and furthermore $\kerr\alpha=\ker\alpha$.
    Likewise, for $\al\in \I_X$  we have $\ker \al=\Delta_{\dom \al}$, so that the conditions regarding kernels are equivalent to those for domains, and only one set needs to be kept. 
\end{proof}

\begin{thm}\label{thm:Ix Tx PTx WLC}
    The monoids $\I_X$, $\T_X$ and $\PT_X$ are all weakly left coherent.
\end{thm}

\begin{proof}
The proof is analogous to, and easier than, that of Theorem \ref{thm:Ix Tx PTx WRC}.
This time, given $\al,\be\in S=\PT_X$, we need to establish existence of $\epsilon\in S$ such that
$S\al\cap S\be=S\epsilon$. This is equivalent to $\im \epsilon=\im \al\cap \im \be$, and it is clear that such $\epsilon$ exists.
The modification to $S=\T_X$ is equally straightforward, but it should be noted that the possibility $S\al\cap S\be=\emptyset$ arises when $\im \al\cap\im \be=\emptyset$.
No modification is needed for $\I_X$.
\end{proof}

\begin{rem}
One can prove the assertions for $\I_X$ in Theorems \ref{thm:Ix Tx PTx WRC} and \ref{thm:Ix Tx PTx WLC}  by using the fact the intersection of any two principal right ideals in an inverse monoid is again principal \cite[Lemma 5.1.6 (2)]{howie95}.
\end{rem}

\begin{thm}\label{thm: Px WRC}
    The partition monoid $\P_X$ is weakly right coherent and weakly left coherent.
\end{thm}
\begin{proof}
    Yet again, by Corollary \ref{pro:regwc}, it is sufficient to show for $\al,\be\in \P_X$ that $\al\P_X\cap \be\P_X$ is finitely generated. 
        
    If $\al\P_X\cap \be\P_X=\emptyset$  we are done. Suppose therefore that there is some $\ga\in \al\P_X\cap \be\P_X$, so that $\ga\P_X\subseteq \al\P_X$ and $\ga\P_X\subseteq \be\P_X$. From \ref{ss:part}, $\ker\alpha\subseteq \ker \gamma$, $N_U(\al)\subseteq N_U(\ga)$, $\ker\beta\subseteq \ker\gamma$ and $N_U(\be)\subseteq N_U(\ga)$. Let $Y$ be the union of the vertices in $N_U(\al)\cup N_U(\be)$. Note that as $N_U(\al)\cup N_U(\be)\subseteq N_U(\ga)$ any block of $N_U(\al)$ either coincides with a block of $N_U(\be)$ or is disjoint from all blocks of $N_U(\be)$, and vice versa. Therefore $N_U(\al)\cup N_U(\be)$ is a partition of $Y$. Also, since $\ga$ exists, there are no edges in $\al$ or $\be$ between $X\backslash Y$ and $Y.$ 
    
    Consider the partition $\Gamma$ of $X\backslash Y$ induced by $(\ker\al\cup\ker\be)|_{X\backslash Y}$. For each part $Z$ in $\Gamma$ choose some $z\in Z.$ Then consider the element $\de\in \P_X$ which has $N_U(\de)=N_U(\al)\cup N_U(\be)$ and the other non-trivial blocks are $Z\cup z^\prime$ for each $Z\in \Gamma$. By construction, $\de\in \al\P_X\cap \be\P_X$, and it is clear from the definition of $\de$ that if $\ga\in \al\P_X\cap \be\P_X$ then $\ga\in \de\P_X$. Hence $\al\P_X\cap \be\P_X=\de\P_X$ and so $\P_X$ is weakly right coherent.

    Weak left coherence of $\P_X$ follows by duality.
\end{proof}


\section{Annihilator congruences}
\label{sec:annihilators}

In this section we prove a number of technical results in the context where we have several elements in a monoid $S$ satisfying some prescribed constraints. We then define a particular right annihilator congruence and show that, under a few assumptions about the chosen elements, it is not finitely generated. The results of this section will enable us to prove our non-coherency results in the next section.
Throughout the section, $\N$ stands for the set $\{1,2,\dots\}$ of natural numbers, and $\N_0=\N\cup\{0\}$.

\begin{prop}\label{prop:Tpres}
The semigroup $P=P(g,h,e)$ defined by the presentation
    \[
   \begin{array}{ccl}
        P(g,h,e)=\langle g,h,e &|&  hg=gh,\  ghg=g,\ hgh=h,\ ghe=egh=e,\ e=ee,\\
         && eg^keh^{k}=g^keh^{k}e, \ eh^keg^{k}=h^keg^{k}e, \ k\in \N\ \rangle.
    \end{array}
    \]
is isomorphic to the inverse submonoid of the symmetric inverse monoid $\I_\Z$ generated by the elements
\[
\gamma: x\mapsto x+1\ (x\in\Z),\ \epsilon: x\mapsto x \ (x\in\Z\setminus\{0\})
\]
via $g\mapsto \gamma$, $h\mapsto \gamma^{-1}$, $e\mapsto \epsilon$.
\end{prop}

\begin{proof}
Let $P'$ denote the inverse submonoid of $\I_\Z$ generated by $\gamma$ and $\epsilon$.
A routine computation verifies that $\gamma$, $\gamma^{-1}$ and $\epsilon$ satisfy the defining relations for $P$.
Hence the above mapping induces an onto homomorphism $\phi:P\rightarrow P'$. We will show that $\phi$ is also injective.

Working in $P$,
it is immediate from the presentation that $gh=hg$ is the identity element. In the remainder of the proof we will write this element as $1$.
The element $h$ is the group inverse of $g$, so we will write $g^{-1}$ instead of $h$.

Since $g^m e g^{-m}g^m e g^{-m}=g^me1eg^{-m}=g^meg^{-m}$
it follows that all $g^m e g^{-m}$ with $m\in\Z$ are idempotents.
These idempotents commute:
\[
	(g^neg^{-n})(g^meg^{-m}) = g^m(g^{n-m}eg^{m-n})e g^{-m} = g^me(g^{n-m}eg^{m-n})g^{-m} = (g^meg^{-m})(g^neg^{-n}).
\]
Every element $u$ of $P$ is equal to a product of the form
$u=g^{i_1}eg^{i_2}eg^{i_3}e\dots eg^{i_{n-1}}eg^{i_n}$ 
with $n\geq 1$, $i_1,i_n\in\N_0$ and $i_j\in\N$ for $2\leq j\leq n-1$, which in turn can be written as:
\[
    u  =  (g^{i_1}eg^{-i_1})(g^{i_1+i_2}eg^{-(i_1+i_2)})\dots (g^{i_1+\dots+i_{n-1}}eg^{-(i_1+\dots +i_{n-1})})g^{i_1+\dots+i_n}.
\]
Using idempotency and commutativity, it follows that $u$ can be written as
\begin{equation}\label{eq:unf}
u=(g^{k_1}eg^{-k_1})\dots (g^{k_p}eg^{-k_p})g^k,
\end{equation}
where $p\geq 0$, and where $k_1,\dots, k_p$ and $k$ are integers with $k_1<\dots <k_p$.

To show that $\phi$ is injective, let $u,v\in P$ be such that $u\phi=v\phi$.
Let $u$ be written as in \eqref{eq:unf}, and write $v$ analogously as
\[
v= (g^{l_1}eg^{-l_1})\dots (g^{l_q}eg^{-l_q})g^l,
\]
with $q\geq 0$, and where $l_1,\dots, l_q$ and $l$ are integers with $l_1<\dots <l_q$.
Now, $u\phi$ is the element $\mu$ of $\I_\Z$ with domain $\Z\setminus\{ -k_1,\dots, -k_p\}$ and $x\mu=x+k$ for all $x$ in that domain.
Likewise, $v\phi$ is the element $\nu$ of $\I_\Z$ with domain $\Z\setminus\{ -l_1,\dots, -l_q\}$ and $x\nu=x+l$ for all $x$ in that domain.
From $\mu=\nu$ we obtain $k=l$, and recalling that $k_1<\dots<k_p$ and $l_1<\dots<l_q$, also that $p=q$ and $k_t=l_t$ for 
$t=1,\dots,p$. It then follows that $u=v$ in $P$, and the proof is complete.
\end{proof}

From the concrete representation of $P$ given in Proposition \ref{prop:Tpres} it is easy to see that
the group of units of $P$ is the infinite cyclic subgroup $G$ generated by $g$. 
The decomposition \eqref{eq:unf} from the proof of Proposition \ref{prop:Tpres} implies that
$P$ is generated by $G$ and the single idempotent $e$. 
Such semigroups were considered by McAlister in \cite{mcalister:1998}. 
It is possible to deduce that $P$ is inverse via this general approach, but the above argument seems more straightforward in this instance. 
In fact, $P$ is a semidirect product of $G$, and the semilattice of idempotents $E(P)$,
with $G$ acting on $E(P)$ by conjugation; this can be derived from the proof of Proposition \ref{prop:Tpres}.

Motivated by $P$, we now list some properties of elements $g,h,e$ in a monoid $S$, where $g$ is regular, $h$ is a generalised inverse of $g$ with $gh=hg$, and $e$ is an idempotent. These properties will transpire to be a barrier to coherence.
Recall  from \ref{ss:ids} that the set of idempotents $E(S)$ of a monoid $S$ is equipped with the partial order $\leq$ defined via
$e\leq f\Leftrightarrow ef=fe=e$. 
The properties are:
    \begin{enumerate}[label=\normalfont(NC\arabic*)]
        \item $hge=e=ehg$;\label{it:a1}
        \item $e(g^neh^{n})=(g^neh^{n})e$ and $e(h^neg^{n})=(h^neg^{n})e$ for all $n\in \N$;\label{it:a2}
        \item         for $m,n\in\N$ we have $g^m e h^m\nleq h^neg^n$ and $h^neg^n \nleq g^meh^m$; furthermore, when $m\neq n$, we also have $g^meh^m\nleq g^neh^n$ and $h^meg^m \nleq h^neg^n$;\label{it:a3}
        \item for all $n\in \N$ we have that $eg^neh^n\not\leq g^keh^k$ for any $0<k<n$ and $eg^neh^n\not\leq h^keg^k$ for any $0<k\leq n$.  \label{it:a4}
    \end{enumerate}
    
    Note that, as in the proof of Proposition \ref{prop:Tpres}, conditions \ref{it:a1}, \ref{it:a2} imply that all the elements appearing in \ref{it:a3}, \ref{it:a4} are idempotents.

 Informally, condition \ref{it:a3} asserts that the conjugates of $e$ by the integer powers of $g$ are all distinct and form an infinite antichain in $E(S)$.
 Condition \ref{it:a4} then gives information about the position of an element of the form $eg^neh^n$
 with respect to a finite portion $\{ g^keh^k,h^keg^k\::\: k=1,\dots,n\}$ of this antichain: it (obviously) lies below $g^neh^n$ and nothing else.

\begin{prop}
\label{pro:PA}
The elements $g,h,e$ of the monoid $P$ from Proposition \ref{prop:Tpres} satisfy all the properties \ref{it:a1}--\ref{it:a4}.
\end{prop}

\begin{proof}
The necessary pre-conditions for $g,h,e$, as well as properties \ref{it:a1} and \ref{it:a2}, follow from the defining presentation for $P$.
To prove that \ref{it:a3} and \ref{it:a4} are satisfied we work in the isomorphic copy $P'$ of $P$ given in Proposition \ref{prop:Tpres}.
From \ref{ss:reg} and \ref{ss:tran} we have that the relation $\eta\leq \zeta$ on $E(\I_\Z)$ is equivalent to $\dom(\eta)\subseteq \dom(\zeta)$. Since
$$\dom(\epsilon\ga^k\epsilon\ga^{-k})=\Z\backslash \{0,-k\}\quad \text{and}\quad \dom(\ga^{k}\epsilon\ga^{-k})=\Z\backslash \{-k\} $$
it follows  that $\{\ga^n\epsilon\ga^{-n}\mid n\in \Z\}$ is an antichain with respect to the natural partial order on $E(\I_\Z)$ and that $\epsilon\ga^n\epsilon\ga^{-n}\leq \ga^k\epsilon\ga^{-k}$ only when $n=k$ or $k=0$. Thus conditions \ref{it:a3} and \ref{it:a4} are satisfied, and the proof is complete.  
\end{proof}

Our main objective in this section is to prove the following result.

\begin{thm}\label{thm:helptheorem}
    Let $S$ be a monoid, let $g\in S$ be regular with $h\in S$ a generalised inverse of $g$ such that $hg=gh$, and let $e\in E(S)$.
    If these elements satisfy \ref{it:a1}--\ref{it:a4}
then $S$ is not right or left coherent.
\end{thm}

\begin{cor}\label{cor:Pforbid}
If a monoid $S$ contains the monoid $P(g,h,e)$ from Proposition \ref{prop:Tpres} as a subsemigroup, then $S$ is not left or right coherent.
\end{cor}

The remainder of this section comprises the proof of Theorem~\ref{thm:helptheorem}. We shall prove that the conditions imply that $S$ is not right coherent; as the conditions are left-right dual we deduce that $S$ is also not left coherent.

The first stage of our proof for Theorem \ref{thm:helptheorem} is the following technical lemma in which we give an explicit generating set for a right annihilator congruence. 

\begin{lemma}\label{lem:gen set annihilator}
    Let $g\in S$ be regular, let $h\in S$ be a generalised inverse of $g$ such that $hg=gh$, and let $e\in E(S)$ be such that \ref{it:a1} and \ref{it:a2} hold. Let $\rho:=\langle(1,g)\rangle$. Then the right annihilator $r(e\rho)$ is given by
    \begin{align}
        \label{eq:rerho1}
        r(e\rho) & =\{(u,v)\in S\times S\mid \textup{for some } n\in\NZ,\ g^neu=ev\ \text{or}\ h^neu=ev\} \\
        \label{eq:rerho2}
        & =\langle \{(1,e)\}\cup \{(g^ne,h^{n}eg^n)\mid n\in \N\} \cup \{(h^ne,g^{n}eh^n)\mid n\in \N\}\rangle.
    \end{align}
\end{lemma}

\begin{proof}
    Note that \ref{it:a1} implies $(hg)^ne=e$ for all $n.$ We remark, under the assumptions on $g,h,e$, that $g^neu=ev$ is equivalent to $eu=h^nev.$ Indeed, suppose $g^neu=ev,$ then
    $$h^{n}(ev) = h^{n}(g^{n}eu) = (hg)^{n}(eu) = eu,$$
    and the converse is analogous.

    From the definition of the right annihilator and Lemma \ref{lem:r cong gen by (1,g)} we have that
    $$r(e\rho)=\{(u,v)\in S\times S\mid \exists m,n\in\NZ,\ g^meu=g^nev\}.$$
  The fact that $g$ and $h$ are inverses of each other in the subsemigroup of $S$ generated by $g,h,e$
    establishes \eqref{eq:rerho1}.

For \eqref{eq:rerho2}, let $\sigma$ denote the set on the right hand side. As $r(e\Delta)\subseteq r(e\rho)$, by Corollary \ref{cor:idem}, we have $(1,e)\in r(e\rho)$. Also, for each $n\in \N$,
    $$h^{n}e(g^ne) = e(h^{n}eg^n),$$ 
    which implies that $(g^ne,h^{n}eg^n)\in r(e\rho)$. By the same argument, $(h^ne,g^neh^n)\in r(e\rho)$, hence $\sigma\subseteq r(e\rho)$.

    For the reverse inclusion suppose that $g^neu=ev,$ for some $n\in \N$. Then
    $$v=1v\ \sigma\ ev = g^neu = (g^ne)eu\ \sigma \ h^{n}eg^neu = eh^{n}eg^neu = eh^{n}ev = eu\ \sigma\ 1u = u.$$
    Analogously, $h^neu=ev$ implies $(u,v)\in\sigma$, and thus $r(e\rho)=\sigma,$ as required.
\end{proof}

The following provides the main step in proving Theorem~\ref{thm:helptheorem}, it is again technical and carries forward several assumptions from the previous lemma.  

\begin{prop}\label{thm:tauN neq tauN+}
    Let $g\in S$ be regular, let $h\in S$ be a generalised inverse of $g$ such that $hg=gh$, and let $e\in E(S)$ be such that 
    \ref{it:a1}--\ref{it:a4} hold. For each $n\in\N$ let 
    \[Y_n=\{(1,e)\}\cup \{(g^ke,h^{k}eg^k)\mid 0<k\leq n\}\cup \{(h^ke,g^{k}eh^k)\mid 0<k\leq n\}.\] 
    Then 
    $\langle Y_{n}\rangle \neq \langle Y_{n+1}\rangle$
    for any $n\in \N$.
\end{prop}

\begin{proof}
    It suffices to show that $(g^ne,h^{n}eg^n)\notin \langle Y_{n-1}\rangle$ for any $n>1$. Aiming for a contradiction we suppose that there is an $n$ for which this is untrue. We note that $g^ne\neq h^{n}eg^n$, since $g^{n}e\ \R\ g^{n}eh^n$ and 
    the elements $g^neh^n$ and $h^neg^n$ are incomparable by \ref{it:a3}. Thus there is a $Y_{n-1}$-sequence of length $m>0$ from $g^ne$ to $h^{n}eg^n.$ Consider such a sequence of minimum length, explicitly: 
    \[g^ne=c_1t_1,\, d_1t_1=c_2t_2,\,\dots,\ d_mt_m=h^{n}eg^n.\]
    where  $m\in \N$,  $t_i\in S$ and $(c_i,d_i)\in Y_{n-1}\cup Y_{n-1}^{-1}$ for $i=1,\dots,m$. 
    
    We shall show by induction that each $(c_i,d_i)$ is equal to $(1,e)$ or $(e,1)$ and that $eg^ne=et_i$ for each $i.$ This contradicts the fact that $d_mt_m=h^{n}eg^n$ as we would then have
    $$eg^neh^n\ \R\ eg^ne = et_m = ed_mt_m = e(h^{n}eg^n) \leq h^{n}e g^n $$
    which contradicts assumption \ref{it:a4}.

    We proceed with the inductive argument. We have $g^ne=c_1t_1,$ so $g^ne\leq_\R c_1$. By assumption \ref{it:a3}, we have $g^neh^n\not\leq g^keh^k$ or $ h^keg^k$ for any $0<k<n$, which implies that $c_1\neq g^ke,h^{k}eg^k,h^ke,g^keh^k$. Therefore $(c_1,d_1)=(1,e)$ or $(e,1).$ It then also follows that $et_1=ec_1t_1=e(g^ne)$ as required.  
    
    For the inductive step we assume that $(c_i,d_i)$ is $(1,e)$ or $(e,1)$ and $et_i=eg^ne$ for each $1\leq i< l.$ Then we have that $d_{l-1}t_{l-1}=c_lt_l$ and we know $d_{l-1}$ is equal to $1$ or $e$; we consider these cases separately. First suppose $d_{l-1}=e$ so $eg^ne=et_{l-1}=c_lt_l$, which implies that $eg^ne\leq_\R c_l.$ Therefore, by assumption \ref{it:a4}, $c_l$ is not  of the form $g^ke,\ h^{k}eg^k,\ h^ke$ or $g^keh^k$ for any $0<k<n$, so that $c_l$ is $1$ or $e.$  Now suppose that $d_{l-1}=1$, so that $t_{l-1}=c_lt_l.$  In this case, if $c_l=g^ke$ then using assumption \ref{it:a2} we have
    $$eg^neh^n\ \R\ eg^ne=et_{l-1}=ec_lt_l=e(g^ke)t_l = (g^{k}h^{k})eg^ket_l=(g^ke)(h^{k}eg^kt_l)\leq_{\R} g^keh^k.$$
    By \ref{ss:reg} this implies $eg^neh^n \leq g^keh^k$, which contradicts assumption \ref{it:a4}. Very similar arguments, using the condition that $eg^neh^n\not\leq h^keg^k,g^keh^k$, demonstrate that $c_l$ is also not $h^{k}eg^k,\ h^ke$ or $g^keh^k$.  Therefore we have $(c_l,d_l)=(1,e)$ or $(e,1).$ Furthermore, as $d_{l-1}=1$ or $e$ we obtain that 
    $$et_l=ec_lt_l=ed_{l-1}t_{l-1}=et_{l-1}=eg^ne.$$
    This completes the inductive step, which completes the proof.
\end{proof}

We are now able to complete the proof of Theorem~\ref{thm:helptheorem}.

\begin{proof}[Proof of Theorem~\ref{thm:helptheorem}]
    Let $\sigma_n=\langle Y_n\rangle$, where $Y_n$ is  defined as in Proposition~\ref{thm:tauN neq tauN+}.  By Lemma \ref{lem:gen set annihilator} we know that $r(e\rho)=\bigcup_{n\in \N} \sigma_n$ and by Theorem \ref{thm:tauN neq tauN+} we know that $\sigma_1\subset \sigma_2\subset\ldots$ is an infinite strictly ascending chain of right congruences. It follows that $r(e\rho)$ is not finitely generated.   
\end{proof}

\begin{rem}
We have seen in Corollary \ref{cor:Pforbid} that the monoid $P(g,h,e)$ is a bar to left or right coherency of over-monoids.
We note that in this regard, the monoid $P(g,h,e)$ is somewhat special, in that it cannot be replaced by an arbitrary non-coherent monoid.
For example, if $F_3$ is the free monoid on 3 generators, then from \cite[Example 6.1]{dandan:2020} the monoid $F_3\times F_3$ is not right coherent, but it embeds into a group, and all groups are right coherent \cite{gould:1992}. 
\end{rem}

\begin{rem} \label{re:f3xf3}
We remark that whereas the conditions in Proposition~\ref{thm:tauN neq tauN+} are sufficient to imply that the monoid is not left or right coherent, they are not necessary. For example, the non-right coherent monoid $F_3\times F_3$ mentioned above has no regular elements, other than the identity. Consequently, there are no suitable elements $e,g,h$ to witness the conditions of Proposition~\ref{thm:tauN neq tauN+}.
An alternative example, this time with non-trivial group of units and a non-unit regular element, may be found in Example 3.1 of \cite{gould:1992}.
In the monoid of this example the two idempotents are comparable, violating the conditions of Proposition~\ref{thm:tauN neq tauN+}.
\end{rem}


\section{Transformation and partition monoids are not right or left coherent}
\label{sec:examples}

In Section \ref{sec:weak for partition} we saw that the monoids $\I_X$, $\T_X$, $\PT_X$ and $\P_X$ are all weakly left and right coherent. We proceed to show that none of these monoids are right or left coherent.

\begin{thm}
\label{thm:coh}
    For any infinite set $X$, the symmetric inverse monoid $\I_X$, the full transformation monoid $\T_X$, the partial transformation monoid $\PT_X$, and the partition monoid $\P_X$ are not right or left coherent.
\end{thm}

\begin{proof}
All the assertions follow from Corollary \ref{cor:Pforbid} and Proposition \ref{prop:Tpres} via a number of embeddings.
First, $\I_\mathbb{Z}$ naturally embeds into $\I_X$, since $X$, being infinite, contains a countable subset.
Next, $\I_X$ is a submonoid of $\PT_X$ by definition. It is also isomorphic to a submonoid of $\P_X$, if we view partial bijections as partitions with trivial kernels and cokernels. Finally, $\PT_X$ embeds into $\T_Y$, where $Y:=X\cup \{-\}$,  via the standard trick of interpreting a partial mapping $\al$ on $X$ as a full mapping on $Y$, which sends to $-$ all $x\in X$ such that $x\al$ is undefined, as well as $-$ itself. But, $X$ is infinite, so $|Y|=|X|$, and thus $\PT_X$ embeds into $\T_X$.
\end{proof}

  We conclude by remarking that the method of demonstrating that a given monoid is not left or right
coherent by virtue of the existence of a subsemigroup with particular properties is a novel one. The subsemigroup $P(g,h,e)$ we devise here for this purpose is inverse
with a non-trivial group of units and infinitely many idempotents. 
As already mentioned in Remark~\ref{re:f3xf3}, the method does not apply \emph{with this} $P(g,h,e)$,
 to some monoids that are known by ad hoc
methods not to be left or right coherent. Furthermore, there are whole classes  of monoids, whose coherency properties
are unknown,  to which the method, as is,  cannot be applied, since all the monoids in questions have only trivial subgroups or only one idempotent.   We therefore pose the question of whether it would be possible to
 find other configurations of elements within monoids that would prevent them from
 being left or right coherent, and such that our method could then be applied  to these classes. 
 A particular case of interest would be that of cancellative monoids, including cancellative monoids of transformations.



\end{document}